\documentclass{article}
\usepackage{amsmath,amsfonts,amsthm,amsopn}

\renewcommand{\ge}{\varepsilon} \newcommand{\R}{\mathbb{R}}
\newcommand{\C}{\mathbb{C}} 
\newcommand{\de}{\partial}
\newcommand{\M}{\mathcal{M}}
\newcommand{\N}{\mathbb{N}}

\DeclareMathOperator{\re}{Re}

\newtheorem{theorem}{Theorem}[section]
\newtheorem{lemma}[theorem]{Lemma}

\newtheorem{prpstn}[theorem]{prpstn}
\newtheorem{corollary}[theorem]{Corollary}

\theoremstyle{definition}
\newtheorem{rmrk}[theorem]{Remark}

\newcommand{\RN}{\mathbb{R}^N}
 
 \newcommand{\Z}{{ \it Z}}
 \newcommand{\pa}{\partial}
 \newcommand{\e}{\varepsilon}
 \newcommand{\p}{\varphi}

 \newcommand{\intr}{\int_{\RN}}
 \newcommand{\en}{\varepsilon_n}

\begin{document}

\title{\bf Multi-peak solutions for magnetic NLS equations without non--degeneracy conditions
\thanks{The first author is supported by MIUR, national project Variational and topological methods in the study of nonlinear
phenomena (PRIN 2005).
The third author is supported by MIUR, national project Variational methods and nonlinear differential equations.}}

\author{S. Cingolani \\
Dipartimento di Matematica, Politecnico di
  Bari \\
  via Orabona 4, I--70125 Bari, Italy \\
\texttt{s.cingolani@poliba.it}
\and L. Jeanjean \\
Equipe de Math\'{e}matiques
(UMR CNRS 6623) \\
 16 Route de Gray, F--25030 Besan\c{c}on, France \\
\texttt{louis.jeanjean@univ-fcomte.fr}
\and S. Secchi \\
Dipartimento di Matematica ed Applicazioni \\
 Universit\`a di Milano--Bicocca \\
 via Cozzi 53, I--20125 Milano, Italy \\
\texttt{Simone.Secchi@unimib.it}}
\date{\today}

\maketitle

\begin{abstract}
In this work we consider the magnetic NLS equation
\begin{equation}\label{eq:XXX}
 \left( \frac{\hbar}{i} \nabla -A(x)\right)^2 u + V(x)u - f(|u|^2)u \, = 0 \, \quad \mbox{ in } \R^N
\end{equation}
where $N \geq 3$, $A \colon \R^N \to \R^N$ is a magnetic potential,
possibly unbounded, $V \colon \R^N \to \R$ is a multi-well electric
potential, which can vanish somewhere, $f$ is a subcritical
nonlinear term. We prove the existence of a semiclassical multi-peak
solution $u\colon \R^N \to \C$ to $(\ref{eq:XXX})$, under conditions
on the nonlinearity which are nearly optimal.
\end{abstract}

\section{Introduction}

We study the existence of a standing wave solution $\psi(x,t) =
\exp(-iEt/\hbar)u(x)$, $E \in \R$,  $u \colon \R^N \to \C$  to the
time--dependent nonlinear Schr\"{o}dinger equation in the presence
of an external electromagnetic field
\begin{equation}\label{eq:1.1}
i\hbar \frac{\de \psi}{\de t} = \left( \frac{\hbar}{i} \nabla
-A(x)\right)^2 \psi +V(x)\psi - f(|\psi|^2)\psi, \quad (t,x) \in \R
\times \R^N.
\end{equation}
Here $\hbar$ is the Planck's constant, $i$ the imaginary unit, $A
\colon \R^N \to \R^N$ denotes a magnetic potential and $V\colon \R^N
\to \R$ an electric potential. This leads us to solve the complex
semilinear elliptic equation
\begin{equation}\label{eq:1.2}
 \left( \frac{\hbar}{i} \nabla -A(x)\right)^2 u + \left( V(x)-E
\right) u - f(|u|^2)u \,  = 0 \, , \quad \ x \in\R^N.
\end{equation}
In the work we are interested to seek for solutions of
$(\ref{eq:1.2})$, which exist for small value of the Planck constant
$\hbar >0$. From a mathematical point of view, the transition from
quantum to classical mechanics can be formally performed by letting
$\hbar \to 0$, and such solutions, which are usually referred to as
{\it semiclassical bound states}, have an important physical
meaning.

For simplicity and without loss of generality, we set $\hbar =
\varepsilon$ and we shift $E$ to $0$. Set  $v(x) = u(\e x)$, $A_\e
(x) = A(\e x)$ and $V_\e(x) = V(\e x),$ equation (\ref{eq:1.2}) is
equivalent to
\begin{equation}\label{eq:1.6}
\left( \frac{1}{i}\nabla -A_\e(x)\right)^2 v + V_\e(x)v - f(|v|^2)v
= 0, \quad x \in \R^N.
\end{equation}
\medskip

In recent years a considerable amount of work has been devoted to
investigating standing wave solutions of (\ref{eq:1.1}) in the case
$A = 0$. Among others we refer to
\cite{FW,O,R,Li,WZ,G,ABC,DF1,DF2,DF3,CNo,CL2,CN,AMS,BW1,BW2,BJ1,BJ2,JT1}.
On the contrary still relatively few papers deal with the case $A
\neq 0$, namely when a magnetic field is present. The first result
on magnetic NLS equations is due to Esteban and Lions.  In
\cite{EL}, they prove the existence of standing waves to
(\ref{eq:1.1}) by a constrained minimization approach, in the case
$V(x) =1$, for $\hbar
>0$ fixed and for special classes of magnetic fields. Successively
in \cite{K}, Kurata showed that equation \eqref{eq:1.6} admits,
under some assumptions linking the magnetic and electric potentials,
a least energy solution and that this solution concentrates near the
set of global minima of $V$, as $\hbar \to 0$. It is also proved
that the magnetic potential $A$ only contributes to the phase factor
of the solution of (\ref{eq:1.6}) for $\hbar >0$ sufficiently small.
A multiplicity result for solutions of (\ref{eq:1.6}) near global
minima of $V$ has been obtained in \cite{C} using topological
arguments. A solution that concentrates as $\hbar \to 0$ around an
arbitrary non-degenerate critical point of $V$ has been obtained in
\cite{CS1} but only for bounded magnetic potentials. Subsequently
this result was extended in \cite{CS2}  to cover also degenerate,
but topologically non trivial, critical points of $V$ and to handle
general unbounded magnetic potentials $A$. If $A$ and $V$ are
periodic functions, the existence of various type of solutions for
$\hbar >0$ fixed has been proved in \cite{AS} by applying minimax
arguments. We also mention the works \cite{BCS, CS} that deal with
critical nonlinearities.

Concerning multi-well electric potentials,  an existence result of
multi-peak solutions to the magnetic NLS equation (\ref{eq:1.6}) is
established  by Bartsch, Dancer and Peng in \cite{BDP}, assuming
that the function $f$ is increasing on $(0, +\infty)$ and satisfies
the Ambrosetti-Rabinowitz's superquadraticity condition. Also, in
\cite{BDP}, an isolatedness condition on the least energy level of
the limiting equation
\[
 - \Delta u +  bu - f(|u|^{2})u = 0, \quad u \in H^1(\R^N, \C)
\]
is required to hold for any $b>0$.

In the present paper we prove an existence result of multi-peak
solutions to (\ref{eq:1.6}), under conditions on $f$,  that we
believe to be nearly optimal. In particular we drop the isolatedness
condition, required in \cite{BDP} and we cover the case of
nonlinearities, which are not monotone.

Precisely, the following conditions will be retained.
\begin{description}
\item[(A1)] $A\colon \R^N\to\R^N$ is of class $C^1$. 
\item[(V1)] $V \in C (\R^N, \R)$, $0 \le V_0 = \inf_{x \in \R^N} V(x)$ and $
\liminf_{|x| \to \infty}V(x)
> 0$.
\item[(V2)] There are
bounded disjoint open sets $O^1,\ldots,O^k$ such that
\[ 0 < m_i = \inf_{x \in O^i} V(x) <    \min_{x \in \pa O^i}
V(x)
\]
for $i =1,\dots ,k$.
\end{description}
For each $i \in \{1,\ldots,k\},$ we define
\[
\M^i = \{x \in O^i \mid V(x) = m_i\}
\]
and we set $\Z = \{x \in \RN \mid V(x) = 0\}$ and $ \displaystyle m
= \min_{i \in \{1, \ldots, k\}}m_i$.

\medskip

On the nonlinearity $f$, we require that
 \begin{description}
 \item[(f0)] $f \colon (0, + \infty) \to \R$ is continuous;
\item[(f1)] $\displaystyle\lim_{t \to 0^+}f(t) = 0$ if $\Z = \emptyset$, and
$\displaystyle \limsup_{t \to 0^+} f(t^2)/t^{\mu} < + \infty$ for
some $\mu >0$ if $\Z \neq \emptyset$;
\item[(f2)]
there  exists some $0 < p < \frac{4}{N-2}, N \geq 3$
                 such that  $\limsup_{t \to + \infty}$ $f(t^2)/t^p$ $<$ $+ \infty$;
\item[(f3)]  there exists $T > 0$ such that
$\frac{1}{2} \hat{m} T^2  < F(T^2),$ where
\[
F(t) = \int_0^t f(s) ds, \quad \hat{m} = \max_{i \in \{1,\ldots,
k\}} m_i.
\]
\end{description}

\medskip
Now by assumption \textbf{(V1)}, we can fix $\widetilde m>0$  such
that
\begin{equation} \label{eq:2.1}
    \widetilde m <\min \left\{ m, \ \liminf_{|x|\to\infty} V(x) \right\}
\end{equation}
and define ${\tilde V}_\e(x) = \max\{ \widetilde m,V_\e(x)\}.$ Let
$H_\e$ be the Hilbert space defined by the completion of
$C_0^\infty(\RN, \C)$ under the scalar product
\begin{equation} \label{eq:2.2}
\langle u,v \rangle_{\e} = \re \int_{\R^N} \left(\frac{1}{i}\nabla u
-
  A_\e (x)u \right) \left(\overline{\frac{1}{i}\nabla v-
  A_\e (x)v}\right)  + {\tilde V}_\e(x) u \overline{v} \ dx
\end{equation}
and $\| \cdot  \|_\e  $  the associated norm.


In the present work, we shall prove the following main theorem.
\begin{theorem}\label{main}
Let $N \geq 3$. Suppose that \textbf{(A)}, \textbf{(V1-2)} and
\textbf{(f0-3)} hold. Then for any $\e>0$ sufficiently small, there
exists a solution $u_\e \in H_\e $ of $(\ref{eq:1.6})$ such that
$|u_\e|$ has $k$ local maximum points $x_\e^i \in O^i$ satisfying
\[
\lim_{\e \to 0}\max_{i=1,\dots,k}\operatorname{dist}(\e x^i_\e,\M^i)
= 0,
\]
 and for
which
\[
|u_{\e}(x)| \leq C_1 \exp \left(- C_2 \, \min_{i=1,...k}|x-
  x_\ge^i| \right)
\]
for some  positive constants $C_1$, $C_2$. Moreover for any sequence
$(\en) \subset (0, \e]$ with $\en \to 0$ there exists a subsequence,
still denoted $(\en)$, such that for each $i \in \{1,\ldots,k\}$
there exist $x^i \in \M^i$ with $\varepsilon_n x_{\e_n}^ i \to x^i,$
a constant $w_i \in \R$ and $U_i \in H^1(\R^N, \R)$ a positive least
energy solution of
\begin{equation}\label{eq:1.4}
 - \Delta U_i +  m_iU_i - f(|U_i|^{2})U_i = 0, \quad U_i \in H^1(\R^N, \R);
\end{equation}
for which one has
\begin{equation}\label{eq:1.5}
 u_{\en}(x) = \sum_{i=1}^k U_i \left({x-x_{\en}^i} \right)e^{i\left(w_i +
 A(x^i)(x-x_{\en}^i) \right)}+ K_n(x)
\end{equation}
where $K_n \in H_{\en}$ satisfies $\|K_n\|_{H_{\en}} = o(1)$ as
$\e_n \to 0$ .
\end{theorem}

\begin{rmrk}\label{regularity}
Arguing as in \cite{clappiturriagaszulkin}, we can develop a
bootstrap argument, and prove that the solution $u_\e \in H_\e$,
found in Theorem \ref{main}, belongs to $C^1(\R^N, \C)$. Indeed, set
$u_\e = v + i w$, with $v, w$ real valued, we have
\[
- \Delta v + V_\e v = G:= f(|u_\e|^2) v - 2 A_\e \cdot \nabla w -
|A_\e|^2 v + (div A_\e) w
\]
and
\[
- \Delta w + V_\e w = H:= f(|u_\e|^2) w + 2 A_\e \cdot \nabla v -
|A_\e|^2 w + (div A_\e) v.
\]
Since $u_\e \in H_\e$, it follows that for each $K$ bounded set in
$\R^N$, $u_\e \in H^1(K, \C)$. Therefore
 $v, w \in H^1(K, \R) \subset
L^{2^*}(K, \R)$ and by $\textbf{(f2)}$, $G, H \in L^{s}(K, \R)$,
where $s = \min\{2^*/(p-1), 2\}$. Standard regularity theory implies
that $v, w \in W^{2,s}(K)$. If $2 s <N $ we can argue as before and
derive  that $v, w \in L^{Ns/(N-2s)} (K, \R)$ and $\nabla v, \nabla
w \in L^{Ns/(N-s)} (K, \R)$. After a finite number of steps, we have
 that $v, w \in W^{2,q}(K)$ for any $q \in [1, + \infty[$ and  by the
Sobolev embedding theorems,  $v,w \in C^{1, \alpha} (K, \R)$, with
$0 < \alpha < 1$.
\end{rmrk}

\begin{rmrk}
If we assume the uniqueness of the positive least energy solutions
of (\ref{eq:1.4}) it is not necessary to pass to subsequences to get
the decomposition (\ref{eq:1.5}) in Theorem \ref{main}.
\end{rmrk}

The proof of Theorem \ref{main} follows the approach which is
developed in \cite{BJ2} to obtain multi-peak solutions when $A = 0.$
Roughly speaking we search directly for a solution of (\ref{eq:1.6})
which consists essentially of $k$ disjoints parts, each part being
close to a least energy solution of (\ref{eq:1.4}) associated to the
corresponding $\M^i$. Namely in our approach we take into account
the shape and location of the solutions we expect to find. Thus on
one hand we benefit from the advantage of the Lyapunov-Schmidt
reduction type approach, which is to discover the solution around a
small neighborhood of a well chosen first approximation. On the
other hand our approach, which is purely variational, does not
require any uniqueness nor non-degeneracy conditions.
\medskip

We remark that differently from \cite{BJ2}, we need to overcome many
additional difficulties which arise for the presence of the magnetic
potential. Indeed it is well known that, in general, there is no
relationship between the spaces $H_\e$ and $H^1(\R^N, \C)$, namely
$H_\e \not \subset H^1(\R^N, \C)$ nor $H^1(\R^N, \C)\not \subset
H_\e$ (see \cite{EL}). This fact explains, for example, the need to
restrict to bounded magnetic potentials $A$ when one uses a
perturbative approach  (see \cite{CS1}). Our Lemma \ref{equi-norm}
and Corollary \ref{equiv-norm-consequence} give some insights of the
relationship between $H_\e$ and $H^1(\R^N, \C)$ which proves useful
in the proof of Theorem \ref{main}. We also answer positively a
question raised by Kurata \cite{K}, regarding the equality between
the least energy levels for the solutions of
\[
- \Delta U + bU = f(|U|^2)U
\]
when $U$ are sought in $H^1(\R^N, \C)$ and $H^1(\R^N, \R)$
respectively. See Lemma \ref{levels} for the precise statement.

\medskip
In contrast to \cite{BDP} we do not treat here the cases $N=1$ and
$N=2$. For such dimensions applying the approach of \cite{BJ2} is
more complex. It can be done when $A = 0$ and for the case of a
single peak (see \cite{BJT}) but it is an open question  if Theorem
\ref{main} still holds when $N=1,2$.
\medskip

The work is organized as follows. In Section 2 we indicate the
variational setting and proves some preliminary results. The proof
of Theorem \ref{main} is derived in Section 3.


\section{Variational setting and preliminary results}
 For any set $B
\subset \R^N$ and $\ge
>0$, let $B_\ge = \{  x \in \R^N \mid \ge x \in B \}$.
\begin{lemma} \label{equi-norm}
Let $K \subset \R^N$ be an arbitrary fixed bounded domain. Assume
that $A$ is bounded on $K$ and $0 < \alpha \leq V \leq \beta$ on $K$
for some $\alpha, \beta >0$. Then, for any fixed $\e \in [0,1],$ the
norm
\[
\|u\|_{K_\e}^2 = \int_{K_\e}\bigg|\left(\frac{1}{i}\nabla -
A_\e(y)\right) u \bigg|^2 + {V}_\e (y) |u|^2 dy
\]
is equivalent to the usual norm on $H^1(K_\e, \C)$. Moreover these
equivalences are uniform, i.e. there exist $c_1, c_2 >0$ independent
of $\e \in [0,1]$ such that
\[
c_1 \|u\|_{K_\e} \leq \|u\|_{H^1(K_\e, \C)} \leq c_2 \|u\|_{K_\e}.
\]
\end{lemma}

\begin{proof}
Our proof is inspired by the one of Lemma 2.3 in \cite{AS}. We have
\[
\int_{K_\e}|A_\e(y)u|^2 dy \leq
\|A\|_{L^{\infty}(K)}\int_{K_\e}|u|^2 dy
\]
and
\[
 \int_{K_\e} {V}_\e(y)|u|^2 dy \leq
\|V \|_{L^{\infty}(K)}\int_{K_\e}|u|^2 dy .
\]
Hence
\begin{eqnarray*}
\int_{K_\e} \left|\frac{1}{i}\nabla u - A_\e(y)u \right|^2 +
{V}_\ge(y)|u|^2 dy &\leq& \int_{K_\e} 2 \big( |\nabla u|^2 +
|A_\e(y)u|^2 \big) +
{V}_\ge (y) |u|^2 dy \\
&\leq& 2 \int_{K_\e} |\nabla u|^2 dy + \big( 2 \|A\|_{L^{\infty}(K)}
+ \|{V} \|_{L^{\infty}(K)}\big) \int_{K_\e}|u|^2 dy.
\end{eqnarray*}
To prove the other inequality note that
\begin{equation*}
\int_{K_\e} \left| \frac{1}{i}\nabla u - A_\e(y)u \right|^2 +
{V}_\ge (y)|u|^2 dy \geq \int_{K_\e} \left| |\nabla u|^2 -
|A_\e(y)u| \right|^2 + {V}_\ge (y) |u|^2 dy.
\end{equation*}
We shall prove that, for some $d >0$ independent of $\e \in [0,1]$,
\begin{equation}\label{eq:2.3}
\int_{K_\e} \big| |\nabla u| - |A_\e(y)u |\big|^2 + {V}_\ge(y)|u|^2
dy \geq d \int_{K_\e}  |\nabla u|^2 + |u|^2 dy.
\end{equation}
Arguing by contradiction we assume that there exist sequences $(\en)
\subset [0,1]$ and $(u_{\en}) \subset H^1(K_{\en}, \C)$ with
$\|u_{\en}\|_{H^1(K_{\en}, \C)} =1$ such that
\begin{equation}\label{eq:2.4}
\int_{K_{\en}} \big| |\nabla u_{\en}| - |A_{\en}(y)u_{\en}| \big|^2
+ V_{\en} (y)|u_{\en}|^2 dy < \frac{1}{n}.
\end{equation}
Clearly $(u_{\en}) \subset H^1(\R^N, \C)$ and
$\|u_{\en}\|_{H^1(\R^N, \C)} =1$. Passing to a subsequence,
$u_{\en}\rightharpoonup u$ weakly in $H^1(\R^N, \C).$ Since $
V_{\en} \geq \alpha >0$ on $K_{\en}$ we see from (\ref{eq:2.4}) that
necessarily
\[
\int_{K_{\en}} |u_{\en}|^2 dy \to 0.
\]
Thus $u_{\ge_n} \to 0$ in $L^2(\R^N, \C)$ strongly and in particular
$u_{\ge_n} \rightharpoonup 0$ in $H^1(\R^N, \C)$. Now
\begin{equation*}
\int_{K_{\en}} \big| |\nabla u_{\en}| - |A_{\en}(y)u_{\en}| \big|^2
dy = \int_{K_{\en}}  |\nabla u_{\en}|^2 - 2|A_{\en}(y)u_{\en}| \,
|\nabla u_{\en}| + |A_{\en}(y)u_{\en}|^2  dy
\end{equation*}
with
\[
\int_{K_{\en}}  |A_{\en}(y)u_{\en}| \, |\nabla u_{\en}| dy \to 0.
\]
Indeed we have
\begin{eqnarray*}
\int_{K_{\en}}  |A_{\en}(y)u_{\en}| \, |\nabla u_{\en}| dy &\leq&
\left( \int_{K_{\en}} |A_{\en}(y)u_{\en}|^2 dy \right)^{\frac{1}{2}}
\bigg( \int_{K_{\en}}  |\nabla u_{\en}|^2 dy \bigg)^{\frac{1}{2}} \\
&\leq& \|A\|_{L^{\infty}(K)} \bigg( \int_{K_{\en}} |u_{\en}|^2 dy
\bigg)^{\frac{1}{2}}.
\end{eqnarray*}
Thus
\[
0 = \limsup_{n \to +\infty} \int_{K_{\en}} \big| |\nabla u_{\en}| -
|A_{\en}(y)u_{\en}| \big|^2 dy \geq \limsup_{n \to +\infty}
\int_{K_{\en}} |\nabla u_{\en}|^2 dy.
\]
But this is impossible since otherwise we would have $u_{\en} \to 0$
strongly in $H^1(\R^N, \C)$.
\end{proof}
From Lemma \ref{equi-norm} we immediately deduce the following
corollary.

\begin{corollary} \label{equiv-norm-consequence}
Retain the setting of ~Lemma \ref{equi-norm}.
\begin{enumerate}
\item[(i)] If $K$ is compact, for any $\e \in (0,1]$ the norm
\[
\|u\|_K^2 := \int_K \left| \left( \frac{1}{i}\nabla - A_\e (y)
\right )u\right|^2 + {V}_\ge (y) |u|^2 dy
\]
is uniformly equivalent to the usual norm on $H^1(K, \C)$.
\item[(ii)] For $A_0 \in \R^N$ and $b>0$ fixed, the norm
\[
\|u\|^2 := \int_{\R^N}\left| \left( \frac{1}{i}\nabla - A_0 \right)
u\right|^2 + b|u|^2 dy
\]
is equivalent to the usual norm on $H^1(\R^N, \C)$.
\item[(iii)] If $(u_{\en}) \subset H^1(\R^N, \C)$ satisfies $u_{\en} =
0$ on $\R^N \setminus K_{\en}$ for any $n \in \N$ and $u_{\en} \to
u$ in $H^1(\R^N, \C)$ then $\|u_{\en} - u\|_{{\en}} \to 0$ as $n \to
\infty.$
\end{enumerate}
\end{corollary}
\begin{proof}
Indeed (i) is trivial, to see (ii) just put $\e = 0$ in Lemma
\ref{equi-norm}. Now (iii) follows from the uniformity of the
equivalence derived in Lemma \ref{equi-norm}.
\end{proof}

\medskip



For future reference we recall the following \textit{Diamagnetic
inequality}: for every $u \in H_\e$,
\begin{equation}\label{eq:2.6}
\left| \left(\frac{\nabla}{i} - A_\e \right) u\right| \geq \big|
\nabla |u|\big|, \quad \mbox{a.e. in $\RN$. }
\end{equation}
See \cite{EL} for a proof. As a consequence of \eqref{eq:2.6}, $|u|
\in H^1(\R^N, \R)$ for any $u \in H_{\e}.$ \medskip

Now we define
\[
\M = \bigcup_{i=1}^k\M^i, \quad O = \bigcup_{i=1}^k O^i
\]
and for any set $B \subset \RN$ and $\alpha > 0,$ $B^\delta = \{x
\in \RN \mid \operatorname{dist}(x,B) \le \delta\}$. For $u \in
H_\e,$ let
\begin{equation}
{\mathcal F}_\ge (u) = \frac{1}{2}\int_{\RN}| D^\ge u|^2 +V_\e |u|^2
dy - \int_{\RN} F(|u|^2) dy
\end{equation}
where we set $D^\ge = (\frac{\nabla}{i} - A_\e).$ Define
\begin{equation*}
\chi_\e(y) =
\begin{cases}
0 &\text{if $y \in O_\e$} \\
\e^{-6 / \mu} &\text{if $y \notin O_\e$},
\end{cases}
\quad \chi^i_\e(y) =
\begin{cases}
0 & \text{if $y \in (O^i)_\e$} \\
\e^{-6 / \mu} &\text{if $y \notin (O^i)_\e$},
\end{cases}
\end{equation*}
and
\begin{equation}
Q_\e(u) = \Big (\intr \chi_\e |u|^2 dy -  1\Big )^{\frac{p+2}{2}}_+,
\ \ Q^i_\e(u) = \Big (\intr \chi^i_\e |u|^2 dy - 1\Big
)^{\frac{p+2}{2}}_+.
\end{equation}
The functional $Q_\e$ will act as a penalization to force the
concentration phenomena to occur inside O. This type of penalization
was first introduced in \cite{BW2}. Finally we define the
functionals $\Gamma_\e,\Gamma_\e^1,\ldots,\Gamma_\e^k : H_\e \to \R$
by
\begin{equation}
\Gamma_\e(u) = {\mathcal F}_\ge(u) + Q_\e(u), \ \ \Gamma_\e^i(u) =
{\mathcal F}_\ge(u) + Q^i_\e(u), \, i = 1,\ldots ,k.
\end{equation}
It is easy to check, under our assumptions, and using the
Diamagnetic inequality (\ref{eq:2.6}), that the functionals
$\Gamma_\e,\Gamma_\e^i \in C^1(H_\e).$ So a critical point of
${\mathcal F}_\ge $ corresponds to a solution of (\ref{eq:1.6}). To
find solutions of (\ref{eq:1.6}) which {\it concentrate} in $O$ as
$\e \to 0,$ we shall look for a critical point of $\Gamma_\e$ for
which $Q_\e$ is zero. \medskip

Let us consider for $a>0$ the scalar limiting equation of
(\ref{eq:1.6})
\begin{equation}\label{eq:2.7}
-\varDelta u + a  u = f(|u|^2) u, \quad u \in H^1(\R^N, \R).
\end{equation}
Solutions of $(\ref{eq:2.7})$ correspond to critical points of the
limiting functional $L_a \colon H^1(\R^N, \R) \to \R$ defined by
\begin{equation}\label{eq:limit}
L_a (u)= \frac12 \int_{\R^N} \big( |\nabla u|^2 + a|u|^2 \big) dy -
\int_{\R^N} F(|u|^2) dy.
\end{equation}
In \cite{BL}, Berestycki and Lions proved that, for any $a>0$, under
the assumptions \textbf{(f0--2)} and \textbf{(f3)} with $\hat{m}=a$,
there exists a least energy solution and that each solution $U$ of
(\ref{eq:2.7}) satisfies the Pohozaev's identity
\begin{equation} \label{eq:2.8}
\frac{N-2}{2}\int_{\RN}|\nabla U|^2 dy + N\int_{\RN}a\frac{|U|^2}{2}
- F(|U|^2) dy = 0.
\end{equation}
From this we immediately deduce that, for any solution~$U$ of
\eqref{eq:2.7},
 \begin{equation} \label{eq:2.9}
\frac{1}{N}\int_{\RN}|\nabla U|^2 dy  = L_a(U).
\end{equation}
We also consider the complex valued equation, for $a>0$,
\begin{equation}\label{eq:2.10}
-\varDelta u + a  u = f(|u|^2) u, \quad u \in H^1(\R^N, \C).
\end{equation}
In turn solutions of $(\ref{eq:2.10})$ correspond to critical points
of the functional $L^c_a :  H^1(\R^N, \C) \to \R$, defined by
\begin{equation}\label{eq:2.11}
L^c_a (v)= \frac12 \int_{\R^N} \left( |\nabla v|^2 + a|v|^2
\right)dy - \int_{\R^N} F(|v|^2)dy.
\end{equation}
In \cite{S} the Pohozaev's identity (\ref{eq:2.8}) and thus
(\ref{eq:2.9}) is given for complex--valued solutions of
(\ref{eq:2.10}). The following result relates the least energy
levels of (\ref{eq:2.7}) and (\ref{eq:2.10}) and positively answers
to a question of  Kurata \cite{K} (see also \cite{SS} for some
elements of proof in that direction). When $N=2$ we say that
\textbf{(f2)} holds if
\medskip
\begin{center}
for all $\alpha>0$ there exists $C_\alpha >0$ such that
$|f(t^2)|\leq C_{\alpha} e^{\alpha t^2}$, for all $t \geq 0$.
\end{center}
\medskip
\begin{lemma}\label{levels}
Suppose that \textbf{(f0--2)} and \textbf{(f3)} with $\hat{m} = a$
hold and that $N\geq 2$. Let $E_a$ and $E_a^c$ denote the least
energy levels corresponding to equations (\ref{eq:2.7}) and
(\ref{eq:2.10}). Then
\begin{equation} \label{eq:uguali}
 E_a= E_a^c.
 \end{equation}
Moreover any least energy solution of (\ref{eq:2.10}) has the form
$e^{i \tau} U$ where $U$ is a positive least energy solution of
(\ref{eq:2.7}) and $\tau \in \R$.
\end{lemma}

\begin{proof}
The inequality $E_a^c \leq E_a$ is obvious and thus to establish
that $E_a^c = E_a$ we just need to prove that $E_a \leq E_a^c$.

We know from \cite{S} that each solution of (\ref{eq:2.10})
satisfies the Pohozaev's identity $P(u)=0$ where $P\colon H^1(\R^N,
\C) \to \R$ is defined by
\[
P(u) = \frac{N-2}{2}\int_{\R^N}|\nabla u|^2 dy +
N\int_{\RN}a\frac{|u|^2}{2} - F(|u|^2) dy.
\]
By Lemma 3.1 of \cite{JT2} we have that
\begin{equation}\label{minimum}
\inf_{\substack{u \in H^1(\R^N,\R) \\ P(u)=0}} L_a(u) = E_a.
\end{equation}
Also it is well known (see for example \cite{HS}) that for any $u
\in H^1(\R, \C)$ one has
\begin{equation}\label{minimum1}
\int_{\R^N}\big|\nabla |u| \big|^2 dy \leq \int_{\R^N}|\nabla u|^2
dy.
\end{equation}
Now let $U$ be a solution of (\ref{eq:2.10}). If $N=2$ we see from
the definition of $P$ that $P(|U|) =0$ and from (\ref{minimum1})
that $L_a(|U|) \leq L_a^c(U)$. Thus  $E_a \leq E_a^c$ follows from
(\ref{minimum}). In addition, if $U$ is a least energy solution of
(\ref{eq:2.10}), necessarily
\begin{equation}\label{ajout}
\int_{\R^N}\big|\nabla |U| \big|^2 dy = \int_{\R^N}|\nabla U|^2 dy
\end{equation}
and $|U|$ is a least energy solution of (\ref{eq:2.7}). If $N \geq
3$ we see from (\ref{minimum1}) that either
\begin{enumerate}
\item[i)] $P(|U|)=0$ and $L_a(|U|) = L_a^c(U).$
\item[ii)] $P(|U|) <0$ and there exists $\theta \in ]0,1[$ such that, for
$U_{\theta}(\cdot) = U(\cdot/\theta)$ we have $P(|U_{\theta}|) = 0$.
Then, since $P(|U_{\theta}|) = 0$, it follows that
\[
L_a(|U_{\theta}|) = \frac{1}{N}\int_{\R^N}\big|\nabla |U_{\theta}|
\big|^2 dy = \frac{\theta^{N-2}}{N} \int_{\R^N} \big|\nabla
|U|\big|^2 dy
\]
and thus
\[
L_a(|U_{\theta}|) < \frac{1}{N}\int_{\R^N}\big|\nabla |U|\big|^2 dy
\leq \frac{1}{N}\int_{\R^N}|\nabla U|^2 dy= L_a^c(U).
\]
\end{enumerate}
In both cases we deduce from  (\ref{minimum}) that $E_a \leq E_a^c$.
In addition if $U$ is a least energy solution of (\ref{eq:2.10})
then (\ref{ajout}) holds and in particular $|U|$ is a least energy
solution of (\ref{eq:2.7}).

\noindent Now, for any $N \geq 2$, let $U$ be a least energy
solution of (\ref{eq:2.10}). Since $|U|$ is a solution of
(\ref{eq:2.7}) we get by elliptic regularity theory and the maximum
principle  that $|U| \in C^1(\R^N, \R)$ and $|U| >0$. At this point,
using (\ref{ajout}), the rest of the proof of the lemma is exactly
the same as the proof of Theorem 4.1 in \cite{HS}.
\end{proof}

\begin{rmrk}
When $N=1$ conditions which assure that (\ref{eq:2.7}) has, up to
translation, a unique positive solution are given in \cite{BL} (see
also \cite{JT3} for alternative conditions). Now following the proof
of Theorem 8.1.6 in \cite{Ca} we deduce that any solution of
(\ref{eq:2.10}) is of the form $e^{i \theta}\rho$ where $\theta \in
\R$ and $\rho >0$ is a solution of (\ref{eq:2.7}). Thus, under the
assumptions of \cite{BL,JT3}, the result of Lemma \ref{levels} also
holds when $N=1$ and the positive least energy solution is unique.
\end{rmrk}

Now let  $S_a$ be the set of least energy solutions $U$ of
(\ref{eq:2.10}) satisfying $$|U(0)| = \max\limits_{y \in \RN}
|U(y)|.$$ By standard regularity any solution of (\ref{eq:2.10}) is
at least $C^1$. Since $f$ is not assumed to be locally H\"{o}lder
continuous we do not know, in contrast to \cite{BDP}, if any least
energy solution is radially symmetric. However the following
compactness result can still be proved.

\begin{prpstn}\label{Prop1}
For each $a > 0$ and $N \geq 3,$ $S_a$ is compact in $H^{1}(\RN,
\C).$ Moreover, there exist $C$, $c > 0$, independent of $U \in
S_a$, such that
\[|U(y)| \le C\exp(-c |y|).\]
\end{prpstn}

\begin{proof}
In \cite{BJ1}, the same results are proved when $S_a$ is restricted
to real solutions. Since, by Lemma \ref{levels}, any least energy
solution of (\ref{eq:2.10}) is of the form $e^{i \tau}\tilde{U}$
with $\tilde{U}$ a least energy solution of (\ref{eq:2.7}) it proves
the lemma.
\end{proof}


\section{Proof of Theorem \ref{main} }

Let
\[
\delta = \frac{1}{10} \min \left\{\operatorname{dist}(\M,\RN
\setminus O), \min_{ i\ne j} \operatorname{dist}(O_i,O_j),
\operatorname{dist}(O,\Z) \right\}.
\]
We fix a $\beta \in (0,\delta)$ and a cutoff $\p \in
C_0^\infty(\RN)$ such that $0 \le \p \le 1,$ $\p(y) = 1$ for $|y|
\le \beta$ and $\p(y) = 0$ for $|y| \geq 2\beta$. Also, setting
$\p_\e(y) = \p(\e y)$  for each $x_i \in (\M^i)^\beta$ and $U_i \in
S_{m_i},$ we define
\[
U_\e^{x_1,\dots,x_k}(y) = \sum_{i=1}^k e^{iA(x_i)(y-
\frac{x_i}{\e})} \p_\e \left(y-\frac{x_i}{\e} \right) U_i
\left(y-\frac{x_i}{\e} \right).
\]
We will find a solution, for sufficiently small $\e
> 0,$  near the set
\[ X_\e = \{U_\e^{x_1\dots,x_k}(y)  \mid \text{$x_i \in (\M^i)^\beta$ and $U_i \in S_{m_i}$ for each $i = 1,\dots,k$}\}.
\]
For each $i \in \{1,\dots,k\}$ we fix an arbitrary $x_i \in \M^i$
and an arbitrary $U_i \in S_{m_i}$ and we define
\[
W^i_\e(y) = e^{iA(x_i)(y -\frac{x_i}{\e})}\p_\e
\left(y-\frac{x_i}{\e}\right) U_i \left(y-\frac{x_i}{\e} \right).
\]
Setting
\[
W^i_{\e,t}(y) = e^{iA(x_i)(y -\frac{x_i}{\e})} \p_\e
\left(y-\frac{x_i}{\e} \right) U_i \left(\frac{y}{t}-\frac{x_i}{\e
t} \right),
\]
 we see that $\lim_{t \to 0}\Vert W^i_{\e,t}
\Vert_\e = 0$ (recall that $N \geq 3$) and that
$\Gamma_\e(W^i_{\e,t}) = {\mathcal F}_\ge(W^i_{\e,t})$ for $t \geq
0$. In the next Proposition we shall prove that there exists $T_i
>0$ such that $\Gamma_\e(W_{\e,T_i}^i) < - 2$ for any $\e >0$ sufficiently
small. Assuming this holds true, let  $\gamma_\e^i(s) = W^i_{\e,s}$
for $s > 0$ and $\gamma_\e^i(0) = 0.$  For $s = (s_1,\dots,s_k) \in
T = [0,T_1] \times \ldots \times [0,T_k]$ we define
\[
\gamma_\e(s) = \sum_{i=1}^kW^i_{\e,s_i} \quad \mbox{ and } \quad
 D_\e = \max_{s \in T } \Gamma_\e (\gamma_\e(s)).
 \]
Finally for each $i \in \{1,\dots,k\},$ let $E_{m_i} = L^c_{m_i}(U)$
for $U \in S_{m_i}$.  In what follows, we set $\displaystyle E_m =
\min_{i \in \{1, \ldots, k\}}E_{m_i}$ and $E = \sum_{i=1}^k E_{m_i}
$. For a set $A \subset H_\e$  and $\alpha
> 0$, we let $A^\alpha = \{u \in  H_\e \ | \   \| u - A \|_\ge \le
\alpha\}.$

\begin{prpstn}\label{prop2} We have
\begin{itemize}
\item[(i)] $\displaystyle \lim_{\e \to 0}D_\e = E,$
\item[(ii)] $\displaystyle  \limsup_{\e \to 0} \max_{s \in
\partial T} \Gamma_\e (\gamma_\e(s)) \le \tilde{E} = \max\{E
- E_{m_i} \ | \ i =1,\dots,k\} < E,$
\item[(iii)] for each $d> 0,$
there exists $\alpha >0$ such that for sufficiently small $\e > 0,$
\[
\Gamma_\e (\gamma_\e (s)) \geq D_{\e}- \alpha \text{ implies that }
\gamma_\e (s) \in X_\e^{d/2}.
\]
\end{itemize}
\end{prpstn}
\begin{proof} Since $\operatorname{supp}(\gamma_\e(s)) \subset \M_\e^{2\beta}$ for each $s
\in T,$ it follows that $\Gamma_\e(\gamma_\e(s))$ $=$ ${\mathcal
F}_\ge(\gamma_\e(s))$ $=$ $ \sum_{i=1}^k {\mathcal F}_\ge
(\gamma_\e^i(s))$. Now, for each $i \in \{1, \dots, k\}$, we claim
that
\begin{equation}\label{eq:3.1}
\lim_{\e \to 0} \int_{\R^N}\left| \bigg( \frac{\nabla}{i} -
A_\e(y)\bigg) W_{\e,s_i}^i\right|^2 dy = s_i^{N-2}\int_{\R^N}|\nabla
U_i|^2 dy.
\end{equation}
Indeed
\begin{equation}\label{inn1}
\int_{\R^N} \bigg| \bigg( \frac{\nabla}{i} - A_\e (y) \bigg)
W_{\ge, s_i}^i \bigg|^2  dy =
 \int_{\R^N} \left( |\nabla W_{\ge, s_i}^i|^2 + |A_\e (y)|^2 |W_{\ge, s_i}^i|^2
 - 2 \re \left[ \frac{1}{i} \nabla W_{\ge, s_i}^i \cdot A_\e( y)
\overline{W_{\ge,s_i}^i}\right] \right) dy
\end{equation}
with
\begin{eqnarray}\label{inn2}
\int_{\R^N}  |\nabla W_{\e, s_i}^i|^2 dy &=&  \int_{\R^N}  \left| i
A (x_i) U_i \left( \frac{y}{s_i} - \frac{x_i}{\e s_i} \right) \p_\e
\left( y- \frac{x_i}{\e} \right) + \frac{1}{s_i} \nabla U_i \left(
\frac{y}{s_i} -
\frac{x_i}{\e s_i} \right) \p_\e \left( y- \frac{x_i}{\e} \right) \right. \notag \\
 && \hspace{3cm}  {} +  \e \nabla_{\tau}
\p \left( \e y- x_i \right) U_i \left(
\frac{y}{s_i} - \frac{x_i}{\e s_i} \right) \bigg|^2 dy \notag \\
&=&  \int_{\R^N} | A(x_i) |^2 \left| U_i \left( \frac{y}{s_i}\right) \right|^2 \left| \p_{\e} \left( y \right)\right|^2 dy \notag \\
&&{}+ \int_{\R^N} \bigg| \frac{1}{s_i}\nabla U_i \left(
\frac{y}{s_i} \right)\p_\e \left( y \right) + \e \nabla_{\tau} \p
\left(\e y \right) U_i \left( \frac{y}{s_i}  \right) \bigg|^2 dy.
\end{eqnarray}
Moreover we have
\begin{equation}\label{inn3}
\int_{\R^N} |A_\ge( y)|^2 |W_{\e, s_i}^i|^2  dy = \int_{\R^N} |A_\e
(y)|^2 \left| U_i \left( \frac{y}{s_i} - \frac{x_i}{\e s_i} \right)
\right|^2 \left| \p_{\e} \left( y- \frac{x_i}{\e} \right)\right|^2
dy
\end{equation}
and
\begin{equation}\label{inn4}
\int_{\R^N}  \re \left[ \frac{1}{i} \nabla W_{\e, s_i}^i \cdot A_\e(
y)\overline{W_{\e,s_i}^i}\right] dy = \int_{\R^N} A_\ge (x_i) \cdot
A_\e (y) \left| U_i \left( \frac{y}{s_i} - \frac{x_i}{\e s_i}
\right) \right|^2 \left| \p_{\e} \left( y- \frac{x_i}{\e}
\right)\right|^2 dy.
\end{equation}
Since, as $\e \to 0$,
\[
\int_{\R^N} |A_\ge (y)|^2  \left| U_i \left( \frac{y}{s_i} -
\frac{x_i}{\e s_i} \right) \right|^2 \left| \p_{\e} \left( y-
\frac{x_i}{\e} \right)\right|^2 dy \to \int_{\R^N} | A(x_i) |^2
\left| U_i \left( \frac{y}{s_i}\right) \right|^2 dy,
\]
and
\[
\int_{\R^N} A_\ge (x_i) \cdot A_\e (y) \left| U_i \left(
\frac{y}{s_i} - \frac{x_i}{\e s_i} \right) \right|^2 \left| \p_{\e}
\left( y- \frac{x_i}{\e} \right)\right|^2 dy \to \int_{\R^N} |
A(x_i) |^2 \left| U_i \left( \frac{y}{s_i}\right) \right|^2 dy,
\]
taking into account $(\ref{inn1})$-$(\ref{inn4})$ it follows that,
\begin{equation}\label{inn6}
\int_{\R^N} \bigg| \bigg( \frac{\nabla}{i} - A_\e (y) \bigg) W_{\ge,
s_i}^i \bigg|^2 dy \to \frac{1}{s_i^2} \int_{\R^N} \left| \nabla U_i
\big(\frac{y}{s_i}\big)\right|^2 dy = s_i^{N-2}\int_{\R^N}|\nabla
U_i|^2 dy
\end{equation}
and this proves (\ref{eq:3.1}). Similarly using the exponential
decay of $U_i$ we have, as $\e \to 0$,
\begin{equation}\label{eq:3.2}
\int_{\R^N}V_\e(y) |W^i_{\e,s_i}|^2 dy \to \int_{\R^N}m_i \, \left|
U_i \left( \frac{y}{s_i} \right) \right|^2 dy = m_i s_i^N
\int_{\R^N}|U_i|^2 dy
\end{equation}
\begin{equation} \label{eq:3.3}
\int_{\R^N} F( |W^i_{\e,s_i}|^2) dy \to \int_{\R^N} F \left(
 \left| U_i \left(\frac{y}{s_i} \right) \right|^2 \right) dy =  s_i^N \int_{\R^N}F(|U_i|^2) dy.
\end{equation}
Thus, from (\ref{eq:3.1}), (\ref{eq:3.2}) and (\ref{eq:3.3}),
\begin{eqnarray*}
{\mathcal F}_\ge(\gamma_\e^i(s_i))  &=&
 \frac{1}{2}\int_{\R^N} \bigg| \bigg(
\frac{\nabla}{i} - A_\e (y) \bigg) \gamma_\e^i(s_i) \bigg|^2 dy  +
V_\e(y)|\gamma_\e^i(s_i)|^2 dy
-\intr F(|\gamma_\e^i(s_i)|^2) dy \\
&=& \frac{s_i^{N-2}}{2}\intr |\nabla U_i|^2 dy + s_i^{N} \intr
\frac{1}{2} m_i|U_i|^2 - F(|U_i|^2) dy + o(1).
\end{eqnarray*}
 Then, from the Pohozaev identity
(\ref{eq:2.8}), we see that
\[
{\mathcal F}_\ge(\gamma_\e^i(s_i)) = \left( \frac{s_i^{N-2}}{2} -
\frac{N-2}{2N}s_i^{N} \right) \intr |\nabla U_i|^2 dy + o(1).
\]
Also
\[
\max_{t \in (0,\infty)} \left( \frac{t^{N-2}}{2} -
\frac{N-2}{2N}t^{N} \right) \intr |\nabla U_i|^2 dy  =E_{m_i}.
\]
At this point we deduce that (i) and (ii) hold. Clearly also the
existence of a $T_i >0$ such that $\Gamma_\e(W_{\e,T_i}^i) <-2$ is
justified. To conclude we just observe that for $g(t) =
\frac{t^{N-2}}{2} - \frac{N-2}{2N}t^{N},$
\[
g^\prime(t)
\begin{cases}
 > 0 &\text{for $t \in (0,1)$},\\
= 0 &\text{for $t = 1$}, \\
< 0 &\text{for $t > 1$},
\end{cases}
\]
and $g^{\prime\prime}(1) = 2- N < 0.$
\end{proof}

Now let
\begin{equation} \label{pathi}
\Phi^i_\e = \{\gamma \in C([0,T_i],H_\e) | \gamma(s_i) =
\gamma^i_\e(s_i) \textrm { for } s_i = 0 \textrm{ or } T_i\}
\end{equation}
and
\[
C^i_\e = \inf_{\gamma \in \Phi^i_\e}\max_{s_i \in [0,T_i]
}\Gamma^i_\e(\gamma(s_i)).
\]
For future reference we need the following estimate.
\begin{prpstn}\label{prop3}
For~$i = 1,\ldots,k$,
\[
\liminf_{\e\to 0}C^i_\e \geq E_{m_i}.
\]
\end{prpstn}
\begin{proof}
Arguing by contradiction, we assume that $\liminf_{\ge \to 0}
C_\ge^i < E_{m_i}$. Then, there exists $\alpha >0$, $\ge_n \to 0$
and $\gamma_n \in \Phi^i_{\ge_n}$ satisfying $\Gamma^i_{\ge_n}
(\gamma_n(s)) < E_{m_i} - \alpha$ for $s \in (0,T_i)$.

We  fix an $\ge_n >0$ such that
\[
\frac{m_i}{2} \ge^{\mu}_n ( 1+ (1+ E_{m_i})^{2/(p+2)} ) <
\min\{\alpha, 1\}
\]
and ${\mathcal F}_{\ge_n}(\gamma_n(T_i)) < -2$ and denote $\ge_n$ by
$\ge$ and $\gamma_n$ by $\gamma$.

Since ${\mathcal F}_\ge(\gamma(0)) =0$ we can find $s_0 \in (0,1)$
such that ${\mathcal F}_\ge(\gamma(s))  \geq -1$ for $s \in [0,s_0]$
and ${\mathcal F}_\ge(\gamma(s_0))  = -1$. Then for any $s \in [0,
s_0]$ we have
\[
Q^i_\ge(\gamma(s)) \leq \Gamma^i_\ge (\gamma(s)) +1 \leq E_{m_i}
-\alpha +1
\]
so that
\[
\int_{\R^N \setminus O_\ge^i} |\gamma(s)|^2 dy \leq \ge^{6/\mu}
\big(1+ (1 + E_{m_i})^{2/(p+2)}\big) \quad \forall s \in [0,s_0].
\]
Now we notice that for any $s \in [0,T_i]$, $|\gamma(s)| \in
H^1(\R^N, \R)$ and by the Diamagnetic inequality  (\ref{eq:2.6})
\begin{equation}\label{diax}
\int_{\R^N} \big| \nabla |\gamma(s)| \big|^2 \, dy \leq \int_{\R^N}
\left| D^\ge \gamma(s) \right|^2 \, dy.
\end{equation}
Then by $(\ref{diax})$ we have that for $s \in [0,s_0]$
\begin{eqnarray}\label{beo}
{\mathcal F}_\ge(\gamma(s)) &=& \frac12 \int_{\R^N} \left| D^\ge
\gamma(s) \right|^2 \, dy + \frac{m_i}{2}  \int_{\R^N}
|\gamma(s)|^2   \, dy - \int_{\R^N} F(|\gamma(s)|^2)\ dy \notag \\
&& {} + \frac{1}{2}
\int_{\R^N} (V_\ge(y) -m_i) |\gamma(s)|^2  \, dy  \notag \\
&\geq& \frac12 \int_{\R^N} \big| \nabla |\gamma(s)| \big|^2 \, dy +
\frac{m_i}{2} \int_{\R^N} |\gamma(s)|^2   \, dy - \int_{\R^N}
F(|\gamma(s)|^2)\ dy  \notag \\
&& {}+ \frac{1}{2}
\int_{\R^N \setminus O_\e^i} (V_\ge(y) -m_i) |\gamma(s)|^2 \, dy \notag \\
&\geq&
 \frac12 \int_{\R^N} \big| \nabla |\gamma(s)| \big|^2 \, dy +
\frac{m_i}{2}  \int_{\R^N} |\gamma(s)|^2   \, dy - \int_{\R^N}
F(|\gamma(s)|^2)\ dy  \notag \\
&& {} - \frac{m_i}{2}
\int_{\R^N \setminus O_\e^i} |\gamma(s)|^2  \, dy \notag \\
&\geq&
 \frac12 \int_{\R^N} \big| \nabla |\gamma(s)|\big|^2 \, dy +
\frac{m_i}{2}  \int_{\R^N} |\gamma(s)|^2   \, dy - \int_{\R^N}
F(|\gamma(s)|^2)\ dy \notag \\
&& {} - \frac{m_i}{2} \ge^{6/\mu} \big(1+ (1 + E_{m_i})^{2/(p+2)}\big) \notag \\
&=& L_{m_i}(|\gamma(s)|)  - \frac{m_i}{2} \ge^{6/\mu} \big(1+ (1 +
E_{m_i})^{2/(p+2)}\big).
\end{eqnarray}
Thus,  $L_{m_i}(|\gamma(s_0)|) < 0$ and recalling that for the
limiting equation (\ref{eq:2.7}) the mountain pass level corresponds
to the least energy level (see \cite{JT2}) we have that
\[
\max_{s \in [0,T_i] }L_{m_i}(|\gamma(s)|) \geq E_{m_i}.
\]
Then we infer that
\begin{eqnarray}\label{beo1}
E_{m_i} - \alpha &\geq& \max_{s \in [0,T_i]} \Gamma^i_\ge
(\gamma(s)) \geq
\max_{s \in [0,T_i]} {\mathcal F}_\ge (\gamma(s)) \notag \\
&\geq& \max_{s \in
[0,s_0]} {\mathcal F}_\ge (\gamma(s)) \notag \\
&\geq& \max_{s \in [0,s_0]} L_{m_i}(|\gamma(s)|)  - \frac{m_i}{2}
\ge^{6/\mu} \left( 1+ (1 + E_{m_i})^{2/(p+2)} \right)
\notag \\
&\geq& E_{m_i} -  \frac{m_i}{2} \ge^{6/\mu} \left( 1+ (1 +
E_{m_i})^{2/(p+2)} \right)
\end{eqnarray}
and this contradiction completes the proof.
\end{proof}

Now we define \[ \Gamma_\e^\alpha = \{ u \in H_\e \ | \ \Gamma_\e(u)
\le \alpha\}.
\]
\begin{prpstn}\label{prop40}
Let $(\e_j)$ be such that $\lim_{j \to \infty}\e_j = 0$ and
$(u_{\e_j}) \in X_{\e_j}^{d}$ such that
\begin{equation}\label{ass}
\lim_{j \to \infty}\Gamma_{\e_j}(u_{\e_j}) \le E \mbox{ and }
\lim_{j \to \infty}\Gamma_{\e_j}^\prime(u_{\e_j}) = 0.
\end{equation}
Then, for sufficiently small $d
> 0,$ there exist, up to a subsequence, $(y^i_j ) \subset \R^N$, $i
=1,\dots,k$, points $x^i\in \M^i$ (which should not be confused with
the points $x_i$ already introduced), $U_i \in S_{m_i}$ such that
\begin{equation}\label{ass1}
\lim_{j \to \infty} |\e_j y^i_j - x^i| = 0  \mbox{ and } \lim_{j \to
\infty}\Vert u_{\varepsilon_j} - \sum_{i=1}^k e^{iA_\e(y^i_j)(\cdot
-y_j^i)}\p_{\e_j}(\cdot - y^i_j) U_i(\cdot - y^i_j)\Vert_{\e_j} = 0.
\end{equation}
\end{prpstn}
\begin{proof} For simplicity we write $\e$
for $\e_j.$ From Proposition \ref{Prop1}, we know that the $S_{m_i}$
are compact. Then there exist $Z_i \in S_{m_i}$ and
$(x_\varepsilon^i) \subset (\M^i)^\beta$, $x^i \in (\M^i)^\beta$ for
$i =1,\ldots,k$, with $x^i_\varepsilon \to x_i$ as $\varepsilon \to
0$ such that, passing to a subsequence still denoted $(u_\e)$,
\begin{equation}
\label{401} \left\| u_\e - \sum_{i=1}^k e^{iA(x^i)(\cdot-
\frac{x^i_\e}{\e})}\p_{\e}(\cdot - x^i_\e/\e)Z_i(\cdot - x^i_\e/\e)
\right\|_\e \le 2d
\end{equation}
for small $\e > 0.$ We set $u_{1,\e} = \sum_{i=1}^k \p_\e(\cdot -
x_\e^i/\e)u_\e$ and $u_{2,\e} = u_\e -u_{1,\e}$. As a first step in
the proof of the Proposition we shall prove that
\begin{equation}
\label{402}\Gamma_\e(u_\e) \geq \Gamma_\e(u_{1,\e}) +
\Gamma_\e(u_{2,\e}) + O(\e).
\end{equation}
Suppose there exist $y_\e \in \bigcup_{i=1}^{k}
B(x^i_\e/\e,2\beta/\e) \setminus B(x^i_\e/\e,\beta/\e)$ and $R>0$
satisfying
\[
\liminf_{\e \to 0}\int_{B(y_\e,R)}|u_\e|^2 dy
> 0
\]
which means that
\begin{equation}\label{4022}
\liminf_{\e \to 0}\int_{B(0,R)}|v_\e|^2 dy
> 0
\end{equation}
where $v_\e(y) = u_\e(y+y_\e)$. Taking a subsequence, we can assume
that $\e y_\e \to x_0 $ with $x_0$ in the closure of
$\bigcup_{i=1}^k B(x^i,2\beta) \backslash B(x^i,\beta)$. Since
(\ref{401}) holds, $(v_\e)$ is bounded in $H_\e$. Thus, since
$\tilde{m} >0,$ $(v_\e)$ is bounded in $L^2(\R^N,\C)$ and using the
Diamagnetic inequality (\ref{eq:2.6}) we deduce that $(v_\e)$ is
bounded in $L^{p+2}(\R^N, \C)$. In particular, up to a subsequence,
$v_\e \to W \in L^{p+2}(\R^N, \C)$ weakly. Also by Corollary
\ref{equiv-norm-consequence} i), for any compact $K \subset \R^N$,
$(v_\e)$ is bounded in $H^1(K, \C)$. Thus we can assume that $v_\e
\to W$ in $H^1(K, \C)$ weakly for any $K \subset \R^N$ compact,
strongly in $L^{p+2}(K,\mathbb{C})$. Because of (\ref{4022}) $W$ is
not the zero function. Now, since $\lim_{\e \to 0}\Gamma'_\e(u_{\e})
=0,$ $W$ is a non-trivial solution of
\begin{equation}\label{3.100}
- \Delta  W  - \frac{2}{i} A(x_0) \cdot \nabla W + |A(x_0)|^2 W +
V(x_0) W = f(|W|^2) W.
\end{equation}
From (\ref{3.100}) and since $W \in L^{p+2}(\R^N, \C)$ we readily
deduce, using Corollary \ref{equiv-norm-consequence} ii) that $W \in
H^1(\R^N, \C).$ \medskip

Let $ \omega(y)= e^{- i A(x_0) y} W(y)$. Then $\omega$ is a non
trivial solution of the  complex-valued equation
\[
-\Delta  \omega + V(x_0) \omega (y) = f(|\omega|^2) \omega.
\]
For $R>0$ large we have
\begin{equation}\label{happy}
\int_{B(0, R)} \left| \left( \frac{\nabla}{i} - A(x_0) \right) W
\right|^2 dy \geq \frac{1}{2}\int_{\R^N} \left|
\left(\frac{\nabla}{i} - A(x_0)\right) W \right|^2 dy
\end{equation}
and thus, by the weak convergence,
\begin{eqnarray}\label{nnn}
\liminf_{\ge \to 0} \int_{B(y_\ge, R)} | D^\ge u_\ge|^2 dy &=&
\liminf_{\ge \to 0} \int_{B(0, R)} \left| \left( \frac{\nabla}{i} -
A_\e(y+ y_\ge) \right)
v_\ge \right|^2 dy \notag \\
&\geq& \int_{B(0, R)} \left| \left( \frac{\nabla}{i} - A(x_0)
\right) W
\right|^2 dy \notag \\
&\geq& \frac{1}{2} \int_{\R^N} \left| \left(\frac{\nabla}{i} -
A(x_0)\right) W \right|^2 dy = \frac{1}{2} \int_{\R^N} |\nabla
\omega |^2 dy.
\end{eqnarray}
Now recalling from \cite{JT2} that $E_a > E_b$ if $a >b$ and using
Lemma \ref{levels} we have $L^c_{V(x_0)} (\omega) \geq E^c_{V(x_0)}
= E_{V(x_0)} \geq E_m$ since $V(x_0) \geq m$. Thus from (\ref{nnn})
and (\ref{eq:2.9}) we get that
\begin{equation} \label{eq:no}
\liminf_{\ge \to 0} \int_{B(y_\ge, R)} | D^\ge u_\ge|^2 dy \geq
\frac{N}{2} L^c_{V(x_0)} (\omega) \geq \frac{N}{2}E_m>0 .
\end{equation}
which contradicts (\ref{401}), provided $d>0$ is small enough.
Indeed, $x_0 \neq x^i$, $\forall i \in \{1,...,k\}$ and the $Z_i$
are exponentially decreasing.
\medskip

Since such a sequence $(y_{\varepsilon})$ does not exist, we deduce
from \cite[Lemma I.1]{L} that
\begin{equation}
\label{404} \limsup_{\e \to 0}\int_{ \bigcup_{i=1}^k B(x_\e^i/\e,
2\beta/\e) \setminus B(x^i_\e/\e, \beta/\e)}|u_\e|^{p+2} dy = 0.
\end{equation}
As a consequence, we can derive using \textbf{(f1)}, \textbf{(f2)}
and the boundedness of $(\|u_\e\|_2)$ that
\[
\lim_{\e \to 0}\intr F(|u_\e|^2)- F(|u_{1,\e}|^2) -F(|u_{2,\e}|^2)
dy =0.
\]
At this point writing
\begin{multline*}
\Gamma_\e(u_\e) = \Gamma_\e(u_{1,\e}) + \Gamma_\e(u_{2,\e})
+\sum_{i=1}^k \int_{B(x^i_\e /\e, 2\beta/\e) \setminus B(x^i_\e /\e,
\beta/\e)}
 \p_\e (y-x^i_\e /\ge)(1- \p_\e (y-x^i/\ge))|D^{\e} u_\e|^2 \\
 {} + V_\e \p_\e (y-x^i_\e/\ge) (1- \p_\e (y-x^i_\e /\ge))|u_\e|^2 dy
 - \intr F(|u_\e|^2)- F(|u_{1,\e}|^2) -F(|u_{2,\e}| ^2) dy + o(1),
\end{multline*}
as $\e \to 0$ this shows that the inequality (\ref{402}) holds. We
now estimate $\Gamma_\e(u_{2,\e})$. We have
\begin{eqnarray}
  \Gamma_\e(u_{2, \e})
\geq {\mathcal F}_\ge(u_{2, \e}) & = & \frac{1}{2}
\int_{\R^N}|D^{\e} u_{2, \e}|^2 + \tilde{V}_\e |u_{2, \e}|^2 dy -
\frac{1}{2}\int_{\R^N}(\tilde{V}_\e - V_\e) |u_{2, \e}|^2dy
\nonumber -\int_{\R^N}F(|u_{2, \e}|^2)
dy  \nonumber \\
& \geq & \frac{1}{2}\Vert u_{2, \e}\Vert_\e^2 -
\frac{\tilde{m}}{2}\int_{\R^N \setminus O_\e^i}|u_{2, \e}|^2 dy -
\int_{\R^N}F(|u_{2, \e}|^2) dy.  \label{405}
\end{eqnarray}
Here we have used the fact that $\tilde V_\e-V_\e=0$ on $O_\e^i$ and
$|\tilde V_\e-V_\e|\leq \widetilde m$ on $\R^N\setminus O_\e^i$.
Because of \textbf{(f1)}, \textbf{(f2)} for some $C_1, C_2
>0$,
\begin{equation*}
\intr F(|u_{2, \e}|^2) dy \leq \frac{\tilde{m}}{4} \intr |u_{2,
\e}|^2 dy + C_1 \intr |u_{2, \e}|^{\frac{2N}{N-2}} dy
 \leq \frac{\tilde{m}}{4} \intr |u_{2, \e}|^2 dy + C_2
 \|u_{2, \e}\|_\e^{\frac{2N}{N-2}}.
\end{equation*}
Since $(u_\e)$ is bounded, we see from (\ref{401}) that $\Vert u_{2,
\e} \Vert_\e \le 4d$ for small $\e
>0$. Thus taking $d>0$ small enough we have
\begin{equation}
\label{406} \frac{1}{2}\|u_{2, \e}\|_\e^2 -  \intr F(|u_{2, \e}|^2)
dy \geq \|u_{2, \e}\|_\e^2 \Big (\frac{1}{4} -C_2(4d)^{4/(N-2)}\Big
) \geq \frac{1}{8} \|u_{2, \e}\|_\e^2.
\end{equation}
Now note that ${\mathcal F}_\ge$ is uniformly bounded in $X_\e^d$
for small $\e
> 0.$ Thus, so is $Q_\e.$ This implies that for some $C > 0,$
\begin{equation} \label{407}
\int_{\RN \setminus O_\e} |u_{2, \e}|^2 dy \le C\e^{6 / \mu}
\end{equation}
and from (\ref{405})-(\ref{407}) we deduce that $\Gamma_\e(u_{2,
\e}) \geq o(1).$
\medskip

Now for $i = 1,\ldots,k,$ we define $u_{1, \e}^i(y) = u_{1,\e}(y)$
for $y \in O_\e^i, $ $u_{1, \e}^i(y) = 0$ for $y \notin O_\e^i$.
Also we set $ W_\e^i (y) = u_{1, \e}^i(y + x^i_\e /\e).$ We fix an
arbitrary $i \in \{1, \ldots, k\}$. Arguing as before, we can
assume, up to a subsequence, that $W_{\e}^i$ converges weakly in
$L^{p+2}(\R^N, \C)$ to a solution $W^i \in H^1(\RN, \C)$ of
\[
- \Delta  W^i - \frac{2}{i} A(x^i) \cdot \nabla W^i + |A(x^i)|^2 W^i
+ V(x^i) W^i = f(|W^i|^2) W^i, \quad y\in \R^N.
\]
We shall prove that $W_\ge^i$ tends to $W^i$ strongly in $H_\ge$.
Suppose there exist $R>0$ and a sequence $(z_\ge)$ with $z_\ge \in
B(x^i_\e /\ge, 2 \beta /\ge)$ satisfying
\[
\liminf_{\ge \to 0} |z_\ge - x^i_\e/\ge| = \infty \quad \hbox{and}
\quad \liminf_{\ge \to 0} \int_{B(z_\ge , R)} |u^{1,i}_\ge|^2 \, dy
>0.
\]
We may assume that $\ge z_\ge \to z^i \in O^i$ as $\ge \to 0$. Then
$\tilde W_\ge^i(y)= W_\ge^i (y + z_\ge)$ weakly converges in
$L^{p+2} (\R^N, \C)$ to $\tilde{W}^i \in H^1(\R^N, \C)$ which
satisfies
\[
- \Delta  \tilde W^i - \frac{2}{i} A(z^i) \cdot \nabla \tilde W +
|A(z^i)|^2 \tilde W^i + V(z^i) \tilde W^i = f(|\tilde W^i|^2) \tilde
W^i, \quad y \in \R^N
\]
and as before we get a contradiction. Then using \textbf{(f1)},
\textbf{(f2)} and \cite[Lemma I.1]{L}  it follows that
\begin{equation}\label{3.101}
\int_{\R^N} F(|W_\ge^i|^2)dy \to \int_{\R^N} F(|W^i|^2)dy.
\end{equation}
Then from the weak convergence of $W_\e^i$ to $W^i \neq 0$ in
$H^1(K, \C)$ for any $K \subset \R^N$ compact we get, for any $i \in
\{1, \ldots, k \}$,
\begin{eqnarray}
\limsup_{\e \to 0}\Gamma_\e(u_{1, \e}^i)
 &\geq& \liminf_{\e \to 0} {\mathcal F}_\ge(u_{1, \e}^i)  \nonumber \\
&\geq& \liminf_{\e \to 0}  \frac{1}{2} \int_{B(0, R)}
\bigg|\left(\frac{\nabla}{i} - A(\e y + x^i)\right) W_{\e}^i\bigg|^2
\notag \\
&&{}+ V(\e y + x^i )|W^i_{\e}|^2 dy -\hfill \intr F(|W_{\e}^i|^2)dy  \nonumber \\
& \geq & \frac{1}{2} \int_{B(0, R)} \bigg|\left(\frac{\nabla}{i} -
A(x^i)\right) W^i\bigg|^2 +V(x^i)|W^i|^2 dy \notag \\
&& \hfill {}- \intr F(|W^i|^2) dy. \label{411}
\end{eqnarray}
 Since these inequalities hold for any $R >0$ we
deduce, using Lemma \ref{levels}, that
\begin{eqnarray}\label{200}
\limsup_{\ge \to 0} \Gamma_\ge (u_{1, \e}^i) & \geq & \frac{1}{2}
\int_{\R^N} \left| \left(\frac{\nabla}{i} - A(x^i)\right)
W^i \right|^2 dy + \frac{1}{2}\int_{\R^N} V(x^i) |W^i|^2 dy \notag \\
&& {} - \int_{\R^N} F(|W^i|^2)dy \nonumber \\
& = & \frac{1}{2} \int_{\R^N} |\nabla  \omega^i |^2  + V(x^i)
|\omega^i|^2 dy - \int_{\R^N} F(|\omega^i|^2)dy  \nonumber \\
& = & L^c_{V(x^i)}(\omega^i) \geq E_{m_i}^c=E_{m_i}
\end{eqnarray}
where we have set $\omega^i(y) = e^{-i A(x^i)y} W^i(y)$. Now by
(\ref{402}),
\begin{equation} \label{412}
\limsup_{\e \to 0} \Big ( \Gamma_\e (u_{2,\e}) + \sum_{i =1}^k
\Gamma_\e (u_{1,\e}^i) \Big ) =  \limsup_{\e \to 0} \Big ( \Gamma_\e
(u_{2,\e}) +
 \Gamma_\e (u_{1,\e}) \Big )
\leq  \limsup_{\e \to 0}\Gamma_\e (u_\e) \leq E = \sum_{i=1}^k
E_{m_i}.
\end{equation}
Thus, since $\Gamma_\e (u_{2, \e}) \geq o(1)$ we deduce from
(\ref{200})-(\ref{412}) that, for any $ i \in \{1, \ldots k \}$
\begin{equation}
\label{413} \lim_{\e \to 0} \Gamma_\e (u_{1, \e}^i) = E_{m_i}.
\end{equation}
Now (\ref{200}), (\ref{413}) implies that $L_{V(x^i)}(\omega^i) =
E_{m_i}$. Recalling from  \cite{JT2} that $E_a >E_b$ if $a >b$ and
using Lemma \ref{levels} we conclude that $x^i \in \M^i$. At this
point it is clear that $W^i(y) = e^{i A(x^i)y}U_i(y - z_i)$ with
$U_i \in S_{m_i}$ and $z_i \in \R^N.$
\smallskip

To establish that $W_\e^i \to W^i$ strongly in $H_\e$ we first show
that $W_\e^i \to W^i$ strongly in $L^2(\R^N, \C)$. Since $(W_\e^i)$
is bounded in $H_\e$ the Diamagnetic inequality (\ref{eq:2.6})
immediately yields that $(|W_\e^i|)$ is bounded in $H^1(\R^N, \R)$
and we can assume that $|W_\e^i| \to |W^i| = |\omega^i|$ weakly in
$H^1(\R^N, \R)$. Now since $L_{V(x^i)}(\omega^i) = E_{m_i}$, we get
using the Diamagnetic inequality, (\ref{3.101}), (\ref{413}) and the
fact that $V \geq V(x^i)$ on $O^i$,
\begin{eqnarray}\label{sopxxxx}
\int_{\R^N} | \nabla \omega^i |^2 dy &+& \int_{\R^N} m_i
|\omega^i|^2 dy -2 \int_{\R^N} F(|\omega^i|^2)dy \notag \\ &\geq&
\limsup_{\ge \to 0}  \int_{\R^N} \left| \left(\frac{\nabla}{i} -
A(\ge y + x^i) \right)
W_\ge^i \right|^2 dy + \int_{\R^N} V(\ge y + x^i) |W_\ge^i|^2  dy \notag \\
&& \qquad {}- 2 \int_{\R^N} F(|W_\ge^i|^2) dy \notag \\
&\geq& \limsup_{\ge \to 0}  \int_{\R^N} \big| \nabla | W_\ge^i|
\big|^2 dy + \int_{\R^N} V(x^i) |W_\ge^i|^2dy - 2 \int_{\R^N}
F(|W_\ge^i|^2) dy \notag \\
&\geq&  \int_{\R^N} \big| \nabla |\omega^i| \big|^2 dy + \int_{\R^N}
m_i |\omega^i|^2 dy - 2 \int_{\R^N} F(|\omega^i|^2)dy .
\end{eqnarray}
But from Lemma \ref{levels} we know that, since
$L_{V(x^i)}(\omega^i) = E_{m_i}$,
$$ \int_{\R^N} \big| \nabla |\omega^i| \big|^2 dy = \int_{\R^N} \big| \nabla \omega^i \big|^2
dy.$$ Thus we deduce from (\ref{sopxxxx}) that
\begin{equation}\label{1000}
\int_{\R^N}V(\ge y + x^i) |W_\ge^i|^2 dy \to \int_{\R^N} V(x^i)
|W^i|^2 dy.
\end{equation}
Thus, since $V \geq V(x^i)$ on $O^i$, we deduce that
\begin{equation} \label{strongly}
W_\e^i \to W^i \mbox{ strongly in } L^2(\R^N, \C).
\end{equation}
From (\ref{strongly}) we easily get that
\begin{equation} \label{strongly0}
\lim_{\ge \to 0}   \int_{\R^N} \left| \left(\frac{\nabla}{i} - A(\ge
y + x^i)\right) W_\ge^i \right|^2 - \left| \left(\frac{\nabla}{i} -
A( x^i)\right) W_\ge^i \right|^2 dy = 0.
\end{equation}
Now, using (\ref{3.101}), (\ref{sopxxxx}) and (\ref{1000}), we see
from (\ref{strongly0}) that
\begin{multline}\label{sopxxx}
\int_{\R^N} \left| \left(\frac{\nabla}{i} - A(x^i)\right)
W^i \right|^2 dy + \int_{\R^N} V(x^i) |W^i|^2 dy  \\
\geq \limsup_{\ge \to 0}  \int_{\R^N} \left| \left(\frac{\nabla}{i}
- A(\ge y  + x^i)\right)
W_\ge^i \right|^2 dy + \int_{\R^N} V(\ge y + x^i) |W_\ge^i|^2  dy  \\
\geq \limsup_{\ge \to 0}  \int_{\R^N} \left| \left(\frac{\nabla}{i}
- A( x^i)\right)W_\e^i \right|^2 dy + \int_{\R^N} V(x^i) |W_\ge^i|^2
\, dy.
\end{multline}
At this point and using Corollary \ref{equiv-norm-consequence} ii)
we have established the strong convergence $W_\e^i \to W^i$ in
$H^1(\R^N, \C)$. Thus we have
\[
u_{1, \e}^i = e^{iA(x^i)(\cdot - x^i_\e / \e) } U_i(\cdot - x^i_\e /
\e - z_i) + o(1)
\]
strongly in $H^1(\R^N, \C)$. Now setting $y_\e^i = x^i_\e/\e +z_i$
and changing $U_i$ to $e^{i A(x^i) z_i}U_i$ we get that
\[
u_{1, \e}^i = e^{i A(x^i) (\cdot - y_\e^i)} U_i(\cdot - y_\e^i) +
o(1)
\]
strongly in $H^1(\R^N, \C)$. Finally using the exponential decay of
$U_i$ and $\nabla U_i$ we have
\[
u_{1, \e}^i = e^{i A_\e(y_\e^i) (\cdot - y_\e^i)} \p_{\e}(\cdot -
y_\e^i)U_i(\cdot - y_\e^i) + o(1).
\]
From Corollary \ref{equiv-norm-consequence} iii) we deduce that this
convergence also holds in $H_\e$ and thus
\[
u_{1,\e} = \sum_{i=1}^k u_{1, \e}^i = \sum_{i=1}^k  e^{i
A_\e(y_\e
^i) (\cdot - y_\e^i)} \p_{\e}(\cdot - y_\e^i)U_i(\cdot -
y_\e^i) + o(1)
\]
strongly in $H_\e$. To conclude the proof of the Proposition, it
suffices to show that $u_{2, \e} \to 0$ in $H_\e$. Since $E \geq
\lim_{\e \to 0} \Gamma_\e(u_\e)$ and $\lim_{\e \to 0}
\Gamma_\e(u_{1, \e})= E$ we deduce, using (\ref{402}) that $\lim_{\e
\to 0} \Gamma_\e(u_{2,\e})= 0$. Now from (\ref{405})-(\ref{407}) we
get that $u_{2, \e} \to 0$ in $H_\e$.
\end{proof}
\begin{prpstn}\label{prop4}
For sufficiently small $d> 0,$ there exist constants $\omega > 0$
and $\e_0
> 0$ such that $|\Gamma_\e^\prime(u)| \geq \omega$ for $u \in
\Gamma^{D_{\e}}_\e \cap (X_\e^{d} \setminus X_\e^{d/2})$ and $\e \in
(0,\e_0).$
\end{prpstn}
\begin{proof} By contradiction, we suppose that for
$d > 0$ sufficiently small such that Proposition \ref{prop40}
applies, there exist $(\e_j)$ with $\lim_{j \to \infty}\e_j = 0$ and
a sequence $(u_{\e_j})$ with $u_{\e_j} \in X_{\e_j}^{d} \setminus
X_{\e_j}^{d/2}$ satisfying $\lim_{j \to
\infty}\Gamma_{\e_j}(u_{\e_j}) \le E$ and $\lim_{j \to
\infty}\Gamma_{\e_j}^\prime(u_{\e_j}) = 0.$  By Proposition
\ref{prop40}, there exist $(y^i_{\e_j}) \subset \R^N$, $i
=1,\ldots,k,$ $x^i\in \M^i$, $U_i \in S_{m_i}$ such that
\begin{equation*}
\lim_{\e_j \to 0} |\e_j y^i_{\e_j} - x^i| = 0,
\end{equation*}
\begin{equation*}
\lim_{\e_j \to 0} \Big\Vert u_{\varepsilon_j} - \sum_{i=1}^k
e^{iA_{\e_j}(y^i_{\e_j}) (\cdot - y^i_{\e_j})} \p_{\e_j}(\cdot -
y^i_{\e_j}) U_i(\cdot - y^i_{\e_j}) \Big\Vert_{\e_j} = 0.
\end{equation*}
By definition of $X_{\e_j}$ we see that $\lim_{\e_j \to
0}\operatorname{dist}(u_{\e_j},X_{\e_j}) = 0.$ This contradicts that
$u_{\ge_j} \not \in X_{\ge_j}^{d/2}$ and completes the proof.
\end{proof}

>From now on we fix a $d>0$ such that Proposition \ref{prop4} holds.

\begin{prpstn}\label{prop7}
For sufficiently small fixed $\e >  0,$  $\Gamma_\e$ has a critical
point $u_{\e} \in X_\e^{d} \cap \Gamma_{\e}^{D_{\e}}.$
\end{prpstn}

\begin{proof}
We can take $R_0>0$ sufficiently large so that $O \subset B(0,R_0)$
and $\gamma_\varepsilon(s) \in H^1_0(B(0, R/\varepsilon))$ for any
$s \in T$, $R > R_0$ and sufficiently small $\varepsilon
>0$.

We notice that by Proposition \ref{prop2} (iii), there exists
$\alpha \in (0, E- \tilde{E})$ such that for sufficiently small $\e
> 0$,
\[
\Gamma_\e (\gamma_\e(s)) \geq D_\e - \alpha \quad \Longrightarrow
\quad \gamma_\e(s) \in X_{\e}^{d/2} \cap H^1_0(B(0,R/\varepsilon)).
\]
We begin to show that for sufficiently small fixed $\e > 0$, and
$R>R_0$, there exists a sequence $(u_n^R) \subset X_{\e}^{d/2} \cap
\Gamma_{\e}^{D_{\e}} \cap H^1_0(B(0,R/\varepsilon))$ such that
$\Gamma'(u_n^R) \to 0$ in $H^1_0 (B(0,R/\varepsilon))$ as $n \to +
\infty$.

Arguing by contradiction, we suppose that for sufficiently small $\e
> 0,$ there exists $a_R(\e) >0$ such that $|\Gamma_\e'(u)| \geq a_R(\e)$
on $X_\e^{d} \cap \Gamma_{\e}^{D_\e} \cap H^1_0
(B(0,R/\varepsilon))$. In what follows any $u \in H^1_0
(B(0,R/\varepsilon))$ will be regarded as an element in
$H_\varepsilon$ by defining $u=0$ in $\R^N \setminus
B(0,R/\varepsilon)$.

Note from Proposition \ref{prop4} that there exists $\omega
> 0,$ independent of $\e > 0,$ such that $|\Gamma_\e^\prime(u)| \geq
\omega$ for $u \in \Gamma_{\e}^{D_{\e}} \cap (X_\e^{d} \setminus
X_\e^{d/2}).$ Thus, by a deformation argument in $H^1_0
(B(0,R/\varepsilon))$, starting from $\gamma_\e$, for sufficiently
small $\e
>0$ there exists a $\mu \in (0,\alpha)$ and a path $\gamma \in
C([0, T], H_\e)$ satisfying
\[
\gamma(s) = \gamma_\e(s) \quad\text{for $\gamma_\e(s) \in
\Gamma_\e^{D_\e - \alpha}$},
\]
\[
\gamma(s) \in  X_\e^{d} \quad \text{ for $\gamma_\e(s) \notin
\Gamma_\e^{D_\e - \alpha}$}
\]
and
\begin{equation} \label{61}
\Gamma_\e(\gamma(s)) < D_\e - \mu, \quad s \in T.
\end{equation}

Let $\psi \in C_0^\infty(\RN)$ be such that $\psi(y) = 1$ for $y \in
O^{\delta},$ $\psi(y) = 0$ for $y \notin O^{2\delta},$ $\psi(y) \in
[0,1]$ and $|\nabla \psi| \le 2/\delta.$ For $\gamma(s) \in X_\e^d,$
we define $\gamma_1(s) = \psi_\e\gamma(s)$ and $\gamma_2(s) =
(1-\psi_\e)\gamma(s)$ where $\psi_\e (y) = \psi(\e y)$.
 Note that
\begin{multline*}
\Gamma_\e(\gamma(s))  =  \Gamma_\e(\gamma_1(s)) +
\Gamma_\e(\gamma_2(s)) +
\intr \bigl(\psi_\e(1-\psi_\e)|D^\e \gamma(s)|^2 +  V_\e \psi_\e(1-\psi_\e)|\gamma(s)|^2 \bigr)dy \\
{} + Q_\e(\gamma(s)) - Q_\e(\gamma_1(s)) - Q_\e(\gamma_2(s)) - \intr
\bigl(F(|\gamma(s)|^2) -  F(|\gamma_1(s)|^2) - F(|\gamma_2(s)|^2)
\bigr) dy + o(1).
\end{multline*}
Since for $A,B \geq 0,$ $(A+B-1)_+ \geq (A-1)_+ + (B-1)_+$ and since
$p+2 \geq 2$ it follows that
\begin{eqnarray*} Q_\e(\gamma(s))
 & =  & \Big (\intr \chi_\e|\gamma_1(s)+\gamma_2(s)|^2 dy - 1\Big )_+^{\frac{p+2}{2}} \\
 & \geq &  \Big (\intr \chi_\e|\gamma_1(s)|^2 dy + \intr \chi_\e |\gamma_2(s)|^2 dy - 1\Big )_+^{\frac{p+2}{2}} \\
 & \geq & \Big (\intr \chi_\e|\gamma_1(s)|^2 dy - 1\Big
)_+^{\frac{p+2}{2}} +
 \Big (\intr \chi_\e|\gamma_2(s)|^2 dy - 1\Big )_+^{\frac{p+2}{2}}\\
 & = & Q_\e(\gamma_1(s)) + Q_\e(\gamma_2(s)).
\end{eqnarray*}
Now, as in the derivation of (\ref{407}), using the fact that
$Q_\e(\gamma(s))$ is uniformly bounded we have, for some $C>0$
\begin{equation}
\label{4007} \int_{\RN \setminus O_\e} |\gamma(s)|^2 dy \le C\e^{6 /
\mu}.
\end{equation} Thus
denoting $p+2 = 2s +(1-s)\frac{2N}{N-2}, s \in (0,1),$ we see from
(f1), (f2), (\ref{4007}) and using the Sobolev inequalities, that
for some $C_1, C_2> 0,$
\begin{eqnarray}
 \int_{\R^N
\setminus O_\e} F(\gamma(s)) dy  & \leq & C_1
\int_{\R^N \setminus O_\e} |\gamma(s)|^2 + |\gamma(s)| ^{p+2} dy  \nonumber \\
& \le & C_1 \int_{\R^N \setminus O_\e} |\gamma(s)|^2 dy \nonumber
\\
 & + &      C_2 \Big (\int_{\R^N \setminus O_\e} |\gamma(s)|^2 dy\Big )^{s}
       \Vert \gamma(s) \Vert_\e^{(1-s)\frac{2N}{N-2}}. \label{408}
\end{eqnarray}
 We deduce that
\begin{equation}
\label{4009} \lim_{\e \to 0} \int_{\R^N \setminus
O_{\e}}F(\gamma(s)) dy = 0.
\end{equation}
Now, as $\e \to 0$,
\begin{eqnarray*}
 \intr |F(\gamma(s)) -  F(\gamma_1(s)) - F(\gamma_2(s))| dy
 &= & \int_{(O^{2\delta})_\e \setminus (O^\delta)_\e}
|F(\gamma(s)) - F(\gamma_1(s))
- F(\gamma_2(s))|dy  \\
&  \leq & \int_{(O^{2\delta})_\e \setminus (O^\delta)_\e}
F(\gamma(s)) + F(\gamma_1(s)) + F(\gamma_2(s)) dy = o(1)
\end{eqnarray*} since
(\ref{4009}) obviously hold when $\gamma(s)$ is replaced by
$\gamma_1(s)$ or $\gamma_2(s)$. Thus, we see that, as $\e \to 0$,
$$ \Gamma_\e(\gamma(s)) \geq \Gamma_\e(\gamma_1(s)) +
\Gamma_\e(\gamma_2(s)) + o(1).$$ Also
\[ \Gamma_\e(\gamma_2(s)) \geq - \int_{\R^N \setminus O_\e}
F(\gamma_2(s)) dy  \geq o(1).\] Therefore it follows that
\begin{equation}
\label{62} \Gamma_\e(\gamma(s)) \geq  \Gamma_\e(\gamma_1(s)) + o(1).
\end{equation} \smallskip

For $i = 1,\cdots,k,$ we define $\gamma_1^{i}(s)(y) =
\gamma_1(s)(y)$ for $y $ $\in$ $(O^i)^{2\delta}_\e,$
$\gamma_1^{i}(s)(y)$ $=$ $0$ for $y \notin (O^i)^{2\delta}_\e.$ Note
that $(A_1 +\cdots +A_n - 1)_+ \geq \sum_{i=1}^n(A_i-1)_+$ for
$A_1,\cdots,A_n \geq 0,$ and that $(p+2) \geq 2.$ Then, we see that,
\begin{equation}
\label{63} \Gamma_\e(\gamma_1(s)) \geq
\sum_{i=1}^k\Gamma_\e(\gamma_1^{i}(s))
                         = \sum_{i=1}^k\Gamma^i_\e(\gamma_1^{i}(s)).
\end{equation}

>From Proposition \ref{prop2} (ii) and since $\alpha \in (0, E-
\tilde{E})$ we get that $\gamma_1^{i} \in \Phi_\e^i$, for all $i \in
\{1, \ldots, k\}$. Thus by Proposition 3.4 in \cite{CR1},
Proposition \ref{prop3}, and (\ref{63}) we deduce that, as $\e \to
0$,
\[\max_{s \in T}\Gamma_\e(\gamma(s)) \geq E + o(1).\] Since $\limsup_{\e \to 0}D_\e \le E$ this contradicts
(\ref{61}). \medskip

Now let $(u_n^R)$ be a Palais-Smale sequence corresponding to a
fixed small $\e > 0$. Since $(u_n^R)$ is bounded in
$H^1_0(B(0,R/\e))$, we can deduce that $u^R_n$ converges, up to
subsequence, strongly to some  $u^R$ in $H^1_0(B(0,R/\e))$ and $u^R$
is a critical point of $\Gamma_\e$ on $H^1_0(B(0,R/\e))$.

Arguing as in Proposition 2 in \cite{BJ12}, we directly derive that
$u^R$  converges strongly to some $u_\ge$ as $R \to + \infty$ and
$u_\e \in X_\e^{d} \cap \Gamma_{\e}^{D_{\e}}$ is a critical point of
$\Gamma_\ge$.
\end{proof}

\noindent{\bf Completion of the Proof for Theorem \ref{main}.} We
see from Proposition \ref{prop7} that there exists $\e_0 > 0$ such
that, for $\e \in (0,\e_0),$ $\Gamma_\e$ has a critical point $u_\e
\in X_\e^d \cap \Gamma_\e^{D_\e}.$ Thus $u_\e$ satisfies
\begin{equation} \label{101}
\bigg( \frac{1}{i} \nabla - A_\e \bigg)^2  u_\e + V_\e u_\e =
f(|u_\e|^2)u_\e - (p+2)\Big(\int\chi_\e |u_\e|^2 dy -
1\Big)_+^{\frac{p}{2}}\chi_\e u_\e \ \textrm{ in } \ \RN.
\end{equation}
Exploiting Kato's inequality (see \cite[Theorem X.33]{RS})
\[
\Delta |u_\e| \geq - Re \bigg( \frac{\bar{u_\e}}{|u_\e|}\bigg(
\frac{\nabla}{i} - A_\e(y)\bigg)^2 u_\e\bigg)
\]
we obtain
\begin{equation} \label{102}
\Delta |u_\e| \geq V_\e |u_\e| - f(|u_\e|^2)|u_\e| +
(p+2)\Big(\int\chi_\e |u_\e|^2 dy - 1\Big)_+^{\frac{p}{2}}\chi_\e
|u_\e| \ \textrm{ in } \ \RN.
\end{equation}
Moreover by Moser iteration \cite{GT} it follows that $(\Vert u_\e
\Vert_{L^\infty})$ is bounded. Now by Proposition \ref{prop40}, we
see that
\[
\lim_{\e \to 0} \int_{\RN \setminus (\M^{2\beta})_\e}|D^\e u_\e|^2 +
\tilde V_\e|u_\e|^2 dy = 0,
\]
and thus, by elliptic estimates (see \cite{GT}), we obtain that
\begin{equation} \label{103}
\lim_{\e \to 0} \Vert  u_\e  \Vert_{L^\infty(\RN \setminus
(\M^{2\beta})_\e)} = 0. \end{equation} This gives the following
decay estimate for $u_\e$ on $\R^N \setminus (\M^{2\beta})_\e \cup
(\Z^\beta)_\e$
\begin{equation} \label{104}
 |u_\e(y)| \le C \exp(-c \operatorname{dist}(y, (\M^{2\beta})_\e \cup
(\Z^{\beta})_\e))
\end{equation}
for some constants $C, c >0$. Indeed from \textbf{(f1)} and
(\ref{103}) we see that
\[
\lim_{\e \to 0}   \|f(|u_\e|^2)\|_{L^\infty (\R^N \setminus
(\M^{2\beta})_\e \cup (\Z^{\beta})_\e)} =0.
\]
Also $\inf \{V_\e(y) | y \notin (\M^{2\beta})_\e \cup
(\Z^{\beta})_\e \} >0$. Thus, we obtain the decay estimate
(\ref{104}) by applying standard comparison principles (see
\cite{PW}) to (\ref{102}). \medskip

If $\Z \neq \emptyset$ we shall need, in addition, an estimate for
$|u_\e|$ on $(\Z^{2\beta})_\e$. Let $\{H^i\}_{i \in I}$ be the
connected components of int$(\Z^{3\delta})$ for some index set $I.$
Note that $\Z \subset \bigcup_{i \in I}H^i$ and $\Z$ is compact.
Thus, the set $I$ is finite. For each $i \in I,$ let $(\phi^i,
\lambda^i_1)$ be a pair of first positive eigenfunction and
eigenvalue of $- \Delta$ on $(H^i)_\e$ with Dirichlet boundary
condition. From now we fix an arbitrary $i \in I$. By elliptic
estimates \cite[Theorem 9.20]{GT} and using the fact that
$(Q_\e(u_\e))$ is bounded we see that for some constant $C>0$
\begin{equation} \label{l-a}
 \Vert  u_\e  \Vert_{L^\infty ((H^i)_\e)} \le C  \e^{3 / \mu}.
\end{equation}
Thus, from \textbf{(f1)} we have, for some $C>0$
\[
\|f(|u_\e|^2) \|_{L^\infty ((H^i)_\e)} \le C \e^3.
\]
Denote $\phi^i_\e(y) = \phi^i(\e y)$. Then, for sufficiently small
$\e > 0$, we deduce that for $y \in \textrm{int}((H^i)_\e)$,
\begin{equation} \label{l-b}
\Delta \phi^i_\e (y) -V_\e(x)\phi^i_\e(y) + f(|u_\e(y)|^2)\phi^i_\e
(y) \le \Big ( C \e^3 - \lambda_1 \e^2 \Big) \phi^i_\e \leq 0 .
\end{equation}
Now, since $\operatorname{dist}(\partial (\Z^{2\beta})_\e,
(\Z^{\beta })_\e) = \beta /\e$, we see from (\ref{104}) that for
some constants $C, c
>0$,
\begin{equation} \label{l-c}
\|u_\e\|_{L^{\infty} (\partial (\Z^{2 \beta})_\e)} \le C \exp (- c
/\e).
\end{equation}

We normalize $\phi^i$ requiring that
\begin{equation} \label{l-d}
\inf_{y \in (H^i)_\e \cap
\partial (\Z^{2 \delta})_\e} \phi^i_\e(y) = C \exp( -c /\e)
\end{equation}
for the same  $C$, $c>0$ as in (\ref{l-c}). Then, we see that for
some $\kappa  > 0$,
\[
\phi^i_\e(y) \le \kappa C \exp(-c/\e), \quad y \in (H^i)_\e \cap
(\Z^{2\beta})_\e.
\]
Now we deduce, using (\ref{l-a}), (\ref{l-b}), (\ref{l-c}),
(\ref{l-d}) and  \cite[B.6 Theorem]{St} that for each $i \in I,$
$|u_\e| \leq \phi^i_\e$ on $(H^i)_\e \cap (\Z^{2 \beta})_\e$.
Therefore
\begin{equation} \label{105}
|u_\e (y)| \leq C \exp (- c /\e) \mbox{ on } (\Z^{2 \delta})_\e
\end{equation}
for some $C, c > 0$. Now (\ref{104}) and (\ref{105}) implies
 that $Q_\e(u_\e) = 0$ for $\e >0$ sufficiently small and thus $u_\e$ satisfies (\ref{eq:1.6}).
Now using Propositions  \ref{Prop1} and \ref{prop40}, we readily
deduce that the properties of $u_\e$ given in Theorem \ref{main}
hold. Here, in (\ref{eq:1.5}) we also use the fact, proved in Lemma
\ref{levels}, that any least energy solution of (\ref{eq:2.10}) has
the form $e^{i \tau} U$ where $U$ is a positive least energy
solution of (\ref{eq:2.7}) and $\tau \in \R$. $\Box$ \medskip

{\bf Acknowlegement:} The authors would like to thank  the referee
for many useful comments which help clarify the paper.

\end{document}